\DeclareSymbolFont{AMSb}{U}{msb}{m}{n} 			
\documentclass[12pt,oneside,noamsfonts]{amsart}

\usepackage[utf8]{inputenx}  
\usepackage[T1]{fontenc}     

\usepackage[USenglish]{babel}

\usepackage{multirow}

\usepackage{stmaryrd}

\usepackage[charter,expert,sfscaled,cal=cmcal]{mathdesign}  
\renewcommand{\in}{\smallin}
\renewcommand{\notin}{\notsmallin}
\renewcommand{\setminus}{\smallsetminus}

\usepackage[scaled=.96,sups]{XCharter}            
\usepackage[kerning]{microtype}

\usepackage{mathtools}       
\usepackage{enumitem}       
\usepackage{url}             

\usepackage{mathscinet}			

\usepackage{todonotes}

{\newlength\figurewidth}
\newcommand\fdash{\settowidth\figurewidth{9}\raisebox{0.6ex}{\makebox[\figurewidth]{\hrulefill}}}

\usepackage{tikz}
\usetikzlibrary{arrows,positioning,fit,shapes,intersections,calc,patterns}

\usepackage{nicefrac}

\newtheorem{theorem}{Theorem}
\newtheorem{proposition}[theorem]{Proposition}
\newtheorem{lemma}[theorem]{Lemma}
\newtheorem{case}{Case}

\newtheorem*{claim*}{Claim}
\newcommand{\claimdone}{\hfill$\blacksquare$\par}

\newtheorem{corollary}[theorem]{Corollary}
\newtheorem{observation}[theorem]{Observation}
\theoremstyle{definition}
\newtheorem{definition}[theorem]{Definition}

\newtheorem{question}[theorem]{Question}

\newcommand{\pw}[1]{\mathcal{P}\left(#1\right)}  
\newcommand{\fin}{\textbf{\textup{fin}}}          
\newcommand{\iter}{\mathbin{*}}          
\newcommand{\os}{\mleft\{ \,}             
\newcommand{\cs}{\, \mright\}}             
\usepackage{mleftright}                       
\newcommand{\card}[1]{\mleft| #1 \mright|}    
\DeclareMathOperator{\non}{non}

\DeclareMathOperator{\cov}{cov}
\DeclareMathOperator{\cof}{cof}
\DeclareMathOperator{\tr}{tr}

\newcommand{\conv}{\textbf{\textup{conv}}}          
\newcommand{\nwd}{\textbf{\textup{nwd}}}          

\newcommand{\conc}{^\smallfrown}           
\renewcommand{\restriction}{\mathbin{\!\upharpoonright}}   

\renewcommand{\mid}{\shortmid}             

\title[Indestructible ultrafilters]{HL ideals and Sacks indestructible ultrafilters}

\author{David Chodounský}
\address{
Institute of Discrete Mathematics and Geometry, Technische Universität Wien (TU Wien) \and
Institute of Mathematics of the Czech Academy of Sciences,
Žitná~25, \linebreak[2] Praha~1, Czech Republic}
\email{chodounsky@math.cas.cz}

\author{Osvaldo Guzmán} 
\address{Centro de Ciencias Matem\'{a}ticas, Universidad Nacional Aut\'onoma de M\'exico, Apartado Postal 61\fdash3, Xangari, 58089, Morelia, Michoac\'{a}n, M\'{e}xico}
\email{oguzman@matmor.unam.mx}

\author{Michael Hrušák}
\address{Centro de Ciencias Matem\'{a}ticas, Universidad Nacional Aut\'onoma de M\'exico, Apartado Postal 61\fdash3, Xangari, 58089, Morelia, Michoac\'{a}n, M\'{e}xico}
\email{michael@matmor.unam.mx}

\subjclass[2010]{Primary: 03E05, 03E17, 03E15, 03E35}
\keywords{Halpern-Läuchli, ideal, ultrafilter, Sacks forcing}

 \thanks{
 First author was partially supported by the Austrian Science Fund (FWF) P33420. 
 Research of the second and third authors was partially supported  by a PAPIIT grant IN104220 and CONACyT grant A1-S-16164.}

\begin{document}

\begin{abstract}
	We study ultrafilters on countable sets and reaping families 
	which are indestructible by Sacks forcing.
	We deal with the combinatorial characterization of such families and 
	we prove that every reaping family of size smaller than the continuum is Sacks indestructible.
	We prove that complements of many definable ideals are Sacks reaping indestructible,
	with one notable exception, the complement of the ideal $\mathcal Z$ of sets of asymptotic density zero. 
	We investigate the existence of Sacks indestructible ultrafilters and prove 
	that every Sacks indestructible ultrafilter is a $\mathcal Z$-ultrafilter.
\end{abstract}

\maketitle


\noindent
Results concerning preservation of ultrafilters in generic extensions 
are a well developed area of set theory of the reals. 
It turns out the most general property of this kind is to be preserved 
(as an ultrafilter base) by Sacks forcing. 
The main question guiding our interest in this topic is existence of such ultrafilters.

\begin{question}[Miller~\cite{Miller-forcing}]
	Does $\mathsf{ZFC}$ prove the existence of a Sacks-in\-de\-struc\-tible ultrafilter?
\end{question}

A negative answer to this question would imply that there is a model $V$ such that 
there is no extension $W \supset V$ with new reals and which 
would contain an ultrafilter on $\omega$ generated by an ultrafilter from~$V$.

\section{Introduction and Notation}

Our notation and terminology is fairly standard, 
we start by giving a brief overview of notions used in this paper. 
We will also give an introduction to the topic and recall 
the most relevant results in the area.

An \emph{ideal} on $\omega$ is a set $\mathcal I \subset \pw{\omega}$ closed under subsets 
and finite unions. We almost always assume that each ideal is proper and contains all finite sets; 
$\fin \subseteq \mathcal I \neq \pw{\omega}$. 
For $\mathcal X \subset \pw{\omega}$ we will denote the dual family by 
$\mathcal X^* = \os \omega \setminus X \mid X \in \mathcal X \cs$. 
A set $\mathcal F \subset \pw{\omega}$ is a \emph{filter} if $\mathcal F^*$ is an ideal. 
We say that $\mathcal B \subseteq \mathcal F$ is a \emph{base of a filter} $\mathcal F$ 
if for each $F \in \mathcal F$ exists $B \in \mathcal B$ such that $B \subseteq F$. 
A filter is an \emph{ultrafilter} if it is a filter maximal with respect to inclusion. 
Given an ideal $\mathcal I$, we define the \emph{coideal} $\mathcal I^+ = \pw{\omega} \setminus \mathcal I$. 
When $\mathcal F = \mathcal I^*$, we also use $\mathcal F^+$ to denote the coideal $\mathcal I^+$. 
We will often identify subsets of $\omega$ with their characteristic function in~$2^\omega$. 
Consequently we will treat subsets of $\pw{\omega}$ also as subspaces of the Cantor space.

If $A, R \in {[\omega]}^\omega$, we say that $R$ \emph{reaps} $A$ if either $R \subset A$ or $R \cap A = \emptyset$. 
We say that $R$ \emph{splits} $A$ if both sets $A \cap R$ and $A \setminus R$ are infinite. 
Given $\mathcal R \subset \pw{\omega}$ we say that $\mathcal R$ is a \emph{reaping family} if 
for every $A \in {[\omega]}^\omega $ exists $R \in \mathcal R$ which reaps $A$. 
The minimal size of a reaping is the cardinal invariant
\[\mathfrak r = \min \os \card{\mathcal R} \mid \mathcal R \text{ is a reaping family} \cs.\]
Notice that every coideal is a reaping family and a filter $\mathcal U$ is an ultrafilter 
iff it has a base which is a reaping family iff every base of $\mathcal U$ is a reaping family. 
The minimal cardinality of an ultrafilter base is the cardinal 
\[\mathfrak u = \min \os \card{\mathcal B} \mid \mathcal B \text{ is a reaping filter base} \cs.\]
We have $\mathfrak r \leq \mathfrak u$. 
For an ultrafilter $\mathcal U$ the \emph{character} $\chi(\mathcal U)$ is defined as the minimal cardinality 
of its base;
\[\chi(\mathcal U) = \min \os \card{\mathcal B} \mid \mathcal B \text{ is a base of } \mathcal U \cs.\]
For an ideal $\mathcal I$ (not necessarily on $\omega$) we denote by $\cof(\mathcal I)$ 
the cofinality of the poset $(\mathcal I, \subseteq)$. 
For an ultrafilter we have $\chi(\mathcal U) = \cof(\mathcal U^*)$.

We will also need some of the numerous other cardinal characteristics of the continuum~\cite{Blass-cardinals}. 
We say that a family of functions $\mathcal D \subseteq \omega^\omega$ is \emph{dominating} 
if for every $f \in \omega^\omega$ there is $g \in \mathcal D$ such that $f < g$. 
The dominating number is defined as
\[\mathfrak d = \min \os \card{\mathcal D} \mid \mathcal D \text{ is a dominating family} \cs.\]
We say that $P \in {[\omega]}^\omega$ is a \emph{pseudo-intersection} of a family $\mathcal X \subset {[\omega]}^\omega$ 
if $P \setminus X$ is finite for each $X \in \mathcal X$. 
The pseudo-intersection number is defined as
\[\mathfrak p = \min \os \card{\mathcal B} \mid \mathcal B \text{ is filter base without any pseudo-intersection} \cs.\]

We say that an ideal $\mathcal I$ is a \emph{P$^+$-ideal} (or a \emph{P$^+$-filter}) if for every system
$\os X_n \in \mathcal I^+ \mid n \in \omega \cs$ there exists 
$Y = \os y_n \in {[X_n]}^{{<}\omega} \mid n \in \omega\cs$ such that $\bigcup Y \in \mathcal I^+$. 
All F$_\sigmaup$ ideals on $\omega$ are P$^+$~\cite{Mazur}. 
If $\mathcal U$ is a P$^+$-ultrafilter, then we say that $\mathcal U$ is a \emph{P-ultrafilter}. 
A filter $\mathcal F$ is a \emph{rare-filter} if for every interval partition $E$ of $\omega$ 
there exist $F \in \mathcal F$ which is a selector for $E$. 
Rare ultrafilters are also called \emph{Q-ultrafilters}. 
An ultrafilter $\mathcal U$ which is simultaneously both a P-ultrafilter and a Q-ultrafilter 
is called \emph{selective} or \emph{Ramsey}. 
Such ultrafilter has the property that for every $c \colon {[\omega]}^2 \to 2$ there is 
$U \in \mathcal U$ such that $c \restriction {[U]}^2$ is constant.
The existence of none of these ultrafilters is provable in $\mathsf{ZFC}$,
this was first proved by Kunen~\cite{Kunen-points} for selective ultrafilters, 
Miller~\cite{Miller-Q-points} for Q-ultrafilters and Shelah~\cite{Wimmers} for P-ultrafilters,
see also~\cite{silver-model}.

The \emph{Katětov order} introduced in~\cite{Katetov} is a powerful tool for classifying relations 
among ideals and filters.
Let $\mathcal I$ be a family of subsets of a countable set $X$ and
$\mathcal J$  be a family of subsets of a countable set $Y$.
We say that a function $f \colon Y \to X$ is a \emph{Katětov morphism}
if $f^{-1}[A] \in \mathcal J$ for every $A \in \mathcal I$.
If there exists such Katětov morphism, we write $\mathcal I \leq_{\mathrm K} \mathcal J$
($\mathcal I$ is \emph{Katětov below} $\mathcal J$).
It is easy to see that the Katětov order $\leq_{\mathrm K}$
is indeed a reflexive and transitive relation. 

Several definable ideals on countable sets
will play an important role in our considerations. 
Let us give here the definitions of the ideals we
will need. 
\begin{itemize}[labelsep=2ex]
	\item[$\fin$] The ideal $\fin$ consists of finite subsets of $\omega$. 
	\item[$\fin \times \fin$]
	The ideal $\fin \times \fin \subset \pw{\omega^2}$ consist of all subsets of
	$\omega \times \omega$ which have only finite intersection with all but finitely many
	columns.
	\item[$\nwd$] The ideal $\nwd$ consists of nowhere dense subsets of the rationals~$\mathbb Q$. 
	\item[$\conv$] The ideal $\conv \subset \pw{\mathbb Q}$ is generated by all converging sequences of rational numbers. 
	\item[$\mathcal Z$] The \emph{density zero} ideal $\mathcal Z$ consist of all $A \subset \omega$
	such $\lim_{n \to \infty} {\frac{\card{A\cap n}}{n}} = 0$.
	Equivalently, $A \in \mathcal Z$ iff 
	\[\lim_{n \to \infty} {\frac{\card{A\cap [2^n,2^{n+1})}}{2^n}} = 0.\]
	\item[$\mathcal{ED}$] The ideal $\mathcal{ED}$ on $\omega^2$ is a sub-ideal of the ideal $\fin \times \fin$;
	it is generated by the columns and by graphs of functions.
	I.e.\ a set is in $\mathcal{ED}$ if the size of its intersection
	with all but finitely columns is bounded by some number $n \in \omega$.
	\item[$\mathcal I_{{\nicefrac{1}{n}}}$] The \emph{summable ideal} $\mathcal I_{{\nicefrac{1}{n}}}$ consists of all sets $A \subset \omega$
	such that \[\sum \os \frac{1}{n+1} \mid n \in A\cs < \infty.\]
	\item[$\mathcal G_{\mathrm c}$] The ideal $\mathcal G_{\mathrm c} \subset \pw{{[\omega]}^2}$ is defined 
	by $E \in \mathcal G_{\mathrm c}$ iff the graph $(\omega, E)$ does not contain 
	an infinite complete subgraph. 
	\item[$\mathcal{SC}$] The ideal $\mathcal{SC} \subset \pw{\omega}$ is generated by SC-sets. 
	A set $A \subset \omega$ is a \emph{SC-set} if for each $n \in \omega$ 
	the set $\os \os a, b\cs \in {[A]}^2 \mid \card{a-b} < n \cs$ is finite.
	Equivalently, $A \in \mathcal{SC}$ if there exists $k \in \omega$ 
	such that for any given $\ell \in \omega$ there are only finitely many 
	intervals $I$ of length $\ell$ such that $k < \card{I \cap A}$. 
\end{itemize}
The ideals $\fin$, $\mathcal{ED}$ and $\mathcal I_{{\nicefrac{1}{n}}}$ are F$_\sigmaup$, $\nwd$, $\mathcal Z$ and  $\mathcal{SC}$ are F$_{\sigmaup\deltaup}$, while $\fin \times \fin$ and $\conv$ are F$_{\sigmaup\deltaup\sigmaup}$ and the ideal $\mathcal G_{\mathrm c}$ is co-analytic.

Let us briefly mention also some ideals on the Cantor space. 
The ideal of meager subsets of $2^\omega$ is denoted~$\mathscr M$, 
the ideal of Lebesgue null subsets is denoted~$\mathscr N$. 
We denote the Lebesgue measure on $2^\omega$ by~$\muup$.
Every ideal on the Cantor space has a naturally associated 
trace ideal.
For $a \subseteq 2^{<\omega}$ define
$\piup(a) = \os x \in 2^\omega \mid \left(\exists^\infty n\in \omega \right) \, x\restriction n \in a \cs$.
For every $a$ the set $\piup(a)$ is G$_\deltaup$ and every G$_\deltaup$ set is of this form.
Trace ideals derived from ideals on the Cantor space or the Baire space were 
defined independently by several authors, notably 
Brendle and Yatabe~\cite{Brendle-Yatabe}, Thümmel~\cite{Thuemmel-thesis}, and
Hrušák and Zapletal~\cite{Hrusak-Zapletal}.
	For an ideal $\mathscr I$ on $2^\omega$ the \emph{trace ideal}
	$\tr(\mathscr I)$ of $\mathscr I$ is defined by $a \in \tr(\mathscr I)$
	iff $\piup(a) \in \mathscr I$.
Only the trace ideal $\tr(\mathscr N)$ will feature in this paper. 
The following notions from~\cite{Piotr-etal} will be useful.
For $a \subseteq 2^{{<}\omega}$ let $M(a)$ be the set of minimal elements of $a$. 
Define $\varphi \colon 2^{{<}\omega} \to \mathbb R$ by 
\[\varphi (a) = \sum \os 2^{-\card{s}} \mid s \in M(a)\cs = 
\sup\os \sum_{s \in b} 2^{-\card{s}} \mid b \text{ is an antichain in } a \cs.\]
Let $\overline\varphi(a) = \lim \os a \setminus 2^{{<}n} \mid n \in \omega\cs$. 
Then  $\muup(\piup(a)) = \overline\varphi(a)$ 
and $\tr(\mathscr N) = \os a \subseteq 2^{{<}\omega} \mid \overline\varphi(a) = 0\cs$.

The Katětov order on the ideals we defined is fairly well understood, 
see~\cite{Hrusak-order_on_ideals,Hrusak-Hernandez*2}. 
In the following diagram the arrows indicate the directions of the 
existing Katětov morphisms.
Moreover, all the provable Katětov relations are indicated.
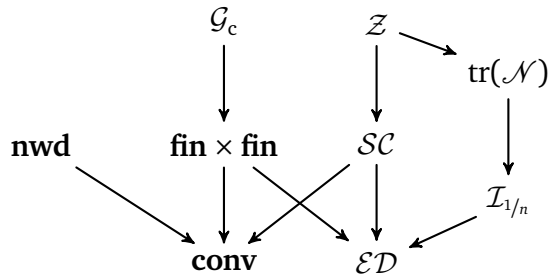
\begin{figure}[h]
\begin{tikzpicture}[node distance=2.5em,line cap=round,rounded corners,thick,>=stealth']
\tikzstyle{bg_arrow}=[-,shorten <=.5em,shorten >=.5em,line width=10pt,draw=white];

  \node (conv)  {$\conv$};
  \node (ED) [right= of conv] {$\mathcal{ED}$};
  \node (finxfin) [above = of conv] {$\fin \times \fin$};
  \node (gc) [above= of finxfin] {$\mathcal G_{\mathrm c}$};
  \node (SC) [above = of ED] {$\mathcal{SC}$};  
  \node (nwd) [left= of finxfin] {$\nwd$};
  \node (sum) [above right=0.4em and 2.2em of ED] {$\mathcal{I}_{\nicefrac{1}{n}}$};
  \node (trN) [above= of sum] {$\tr(\mathscr N)$};
  \node (dens) [above= of SC] {$\mathcal{Z}$};
  
\begin{scope}[<-]
  \path
  (conv) edge (nwd)
  (conv) edge (finxfin)
  (conv) edge (SC)
  (ED) edge (sum)
  (ED) edge (SC)
  (finxfin) edge (gc)
  (SC) edge (dens)
  (sum) edge (trN)
  (trN) edge (dens);
  \path[bg_arrow]
  ;
  \path
  (ED) edge (finxfin);
\end{scope}
\end{tikzpicture}
\end{figure}


Given an ideal $\mathcal I$ on a set $X$, Baumgartner~\cite{Baumgartner-ultrafilters} defined the notion 
of an $\mathcal I$-ultrafilter. This notion can be easily formulated using the Katětov order. 
An ultrafilter on $\omega$ is an \emph{$\mathcal I$-ultrafilter} iff 
$\mathcal I \nleq_{\mathrm K} \mathcal U^*$. 
Many standard properties of ultrafilters can be expressed in this way, 
e.g.\ selective ultrafilters are exactly $\mathcal{ED}$-ultrafilters, 
P-ultrafilters are exactly $\fin\times\fin$-ultrafilters, and equivalently $\conv$-ultrafilters and so on, 
the paper~\cite{Bendle-Flaskova} contains a nice overview. 
The existence these classes of ultrafilters is typically not provable in~$\mathsf{ZFC}$,
some of the strongest results is this direction are the consistency that there 
might be no $\nwd$-ultrafilters~\cite{Shelah-no-nwd-U}, 
and that there is a model with no $\mathcal I$-ultrafilters for any F$_\sigmaup$ ideal~\cite{Cancino}. 
On the other hand, there are ideals $\mathcal I$ for which the existence of 
$\mathcal I$-ultrafilters is provable in~$\mathsf{ZFC}$, see~\cite{Guzman-Hrusak-pospisil}. 

The phenomenon of central interest of this paper is destructibility and indestructibility of ultrafilters. 
Let $V$ be model of set theory, $\mathcal U \in V$ be an ultrafilter on $\omega$, and let 
$W$ be some extension of $V$. If $2^\omega \cap V = 2^\omega \cap W$, then nothing interesting is going on 
(from this point of view) 
and $\mathcal U$ is an ultrafilter in $W$. On the other hand, 
if $W$ contains new reals, $\mathcal U$ is no longer closed with respect to supersets in $W$ 
and is not even a filter. 
Therefore we are rather interested in the filter generated by $\mathcal U$ in $W$.
If $\mathcal U$ generates an ultrafilter in $W$, we say that $\mathcal U$ is preserved in the extension, 
otherwise we say that it is destroyed. 
The ultrafilter $\mathcal U$ is preserved in $W$ if and only if $\mathcal U$ is a reaping family 
in $W$.

Given a forcing $\mathbf P$, we say that an ultrafilter $\mathcal U \in V$ is 
\emph{$\mathbf P$-indestructible} if $\mathcal U$ is preserved in every generic extension of $V$ via $\mathbf P$. 
Otherwise, if $\mathcal U$ is always destroyed, we say that $\mathbf P$ \emph{destroys} $\mathcal U$. 
If the forcing notion $\mathbf P$ does not add reals, then every ultrafilter is $\mathbf P$-indestructible.
If $\mathbf P$ adds an independent real, that is when $V \cap \pw{\omega}$ is not a reaping
family in the generic extension, then there are no $\mathbf P$-indestructible ultrafilters. 
Miller~\cite{Miller(s)} noticed that if $\mathcal U \leq_{\mathrm K} \mathcal V$ are ultrafilters 
and $\mathcal V$ is $\mathbf P$-indestructible, then so $\mathcal U$.
An ideal $\mathcal I$ was constructed in~\cite{Bartoszyski_etal} 
such that whenever an ultrafilter $\mathcal U$ is disjoint with $\mathcal I$, 
then $\mathcal U$ is destroyed in every extension which contains new reals. 
We will improve this result by proving Theorem~\ref{thm:Z-isHL} which states that the ideal of 
density zero sets $\mathcal Z$ is also such an ideal.

P-ultrafilters enjoy a special position among ultrafilters when the question of indestructibility arises. 
Baumgartner and Laver~\cite{Baumgartner-Laver} proved that selective ultrafilters are Sacks forcing-indestructible. 
Later Miller~\cite{Miller-forcing} proved that P-ultrafilters are precisely Miller forcing-in\-de\-structible ultrafilters, 
and Blass noticed that this implies that P-ultrafilters are also Sacks-indestructible. 
In fact, it turns out that Sacks-in\-des\-tructibi\-lity is provably the weakest among these properties~\cite{Miller(s)}, 
see Theorem~\ref{thm:HL-char} of this paper. 
If a forcing $\mathbf P$ adds an unbounded real (i.e.\ $V \cap \omega^\omega$ is not a dominating family in the generic extension),  
then every $\mathbf P$-indestructible ultrafilter has to be a P-ultrafilter. 
If the forcing $\mathbf P$ is $\omega^\omega$-bounding (i.e.\ does not add an unbounded real) 
and $\mathcal U$ is a $\mathbf P$-indestructible ultrafilter, then $\mathcal U \times \mathcal U$ 
is also a $\mathbf P$-indestructible ultrafilter and not a P-ultrafilter, see~\cite{Miller-forcing}.
Let us also remark that if $\mathbf P$ is a proper forcing and 
$\mathcal U$ is a $\mathbf P$-indestructible ultrafilter, 
then $\mathcal U$ generates a P-ultrafilter in the generic extension. 
Moreover, preserving P-ultrafilters is a property which behaves well with 
countable support iteration of proper posets.  

Preservation of ultrafilters and especially P-ultrafilters is extensively studied in the literature. 
Zapletal proved~\cite{Zapletal-preservingP} that for proper definable forcing notions preserving 
P-ultrafilters is equivalent to the weak Laver property and not adding independent reals. 
For general posets just the forward implication needs to hold, 
see the paper of Zapletal for the precise formulation of the result. 
Similarly, preserving selective ultrafilters is for definable posets equivalent to 
the $\omega^\omega$-bounding property and not adding independent reals, see~\cite{Zapletal-idealized}. 
When it comes to destroying ultrafilters, the harm may be of varying kind. 
A forcing $\mathbf P$ diagonalizes an ultrafilter $\mathcal U$ if it adds a pseudo-intersection 
of $\mathcal U$. 
Mildenberger~\cite{Mildenberger-diagonalizing} showed that consistently
there is a forcing notion $\mathbf P$ which diagonalizes certain ultrafilter 
while it preserves an P-ultrafilter at the same time.

This paper will deal extensively with subtrees and subsets of the binary tree $2^{{<}\omega}$. 
We adopt a fairly standard terminology. 
A \emph{tree} $p$ will typically be an initial subset of $(2^{{<}\omega}, \subset)$ without maximal elements. 
For $n \in \omega$ the \emph{level} $n$ of $p$ is the set $p \cap 2^n$. 
For $s \in p$ we denote $p_s = \{ t \in p \mid t \text{ is compatible with } s \}$. 
For $A \subset \omega$ we let $p \restriction A = \os s \in p \cap 2^n \mid n \in A \cs$. 
The set of all branches of $p$ is denoted 
$[p] = \os x \in 2^\omega \mid x \restriction n \in p \text{ for all } n \in \omega \cs$. 
For $s \in 2^{{<}\omega}$ we let $[s] = \os x \in 2^\omega \mid s \subset x\cs$. 

A tree $p \subseteq 2^{{<}\omega}$ is a \emph{perfect tree} if for every 
$s \in p$ there exist $s_0, s_1 \in p_s$ such that $s_0$ and $s_1$ are incompatible. 
A node $s \in p$ is a \emph{stem} of $p$ if $s$ is the maximal node such that $p = p_s$. 
The set of all perfect subtrees of the binary tree is denoted by $\mathbf S$. 
The set $\mathbf S$ equipped with inclusion order is called the \emph{Sacks forcing}. 
This forcing notion was introduced by Sacks~\cite{Sacks}. 
The forcing adds a generic real defined as $g = \bigcap \os [p] \mid p \in G \cs$ 
where $G$ is the generic filter on $\mathbf S$. 
Every generic extension $V[G]$ via the Sacks forcing 
has a minimal degree of constructibility; whenever $W$ is a model of $\mathsf{ZFC}$ 
such that $V \subseteq W \subseteq V[G]$, then either $W = V$ or $W = V[G]$. 
Another prominent property of the Sacks forcing is the so called Sacks property. 
As we will actually not use the definition of this property, 
let us just state that the Sacks property of a given forcing implies that 
the forcing is $\omega^\omega$-bounding, we refer the reader to~\cite{Sacks} for details. 

Our terminology will prominently reference the Halpern--Läuchli theorem for trees. 
However, only one special consequence of the full Halpern--Läuchli theorem 
is relevant for our work.
\begin{proposition}\label{prop:HLthm}
	Given a partition $c \colon 2^{{<}\omega} \to 2$, there exists a perfect tree 
	$p \subseteq 2^{{<}\omega}$ and an infinite set $A \in {[\omega]}^\omega$ 
	such that $c$ is constant on $p \restriction A$.
\end{proposition}
Let us remark that the full theorem is much stronger than the proposition. 
Since the formulation of the full Halpern--Läuchli theorem would need 
further notions irrelevant for this paper, we opt for only giving a reference to the 
original paper~\cite{HL_thm-paper} and for a paper of Laver~\cite{Laver-product_trees} 
which also treats this result.

\section{Halpern--Läuchli families}

Olga Yiparaki called ultrafilters with the following property hlt-ultrafilters 
in her thesis.

\begin{definition}[Yiparaki~\cite{Yiparaki}]
	A family $\mathcal R \subset {[\omega]}^\omega$ is called 
	a \emph{Halpern--Läuchli family} 
	if for every $c \colon 2^{{<}\omega} \to 2$ there are 
	$p \in \mathbf S$ and $A \in \mathcal R$ such that 
	$c$ is constant on $p \restriction A$.
\end{definition}

The name of Halpern--Läuchli families comes from the consequence 
of the Halpern--Läuchli theorem we stated as Proposition~\ref{prop:HLthm}. 

\begin{proposition}\label{prop:HL_fin}
	The coideal $\fin^+$ is a Halpern--Läuchli family.
\end{proposition}

When dealing with an ideal $\mathcal I \subseteq \pw{\omega}$, 
we will say that $\mathcal I$ is HL as a shortcut for 
the statement that $\mathcal I^+$ is a Halpern--Läuchli family. 
I.e.\ Proposition~\ref{prop:HL_fin} can be phrased `$\fin$ is a HL ideal.'

By considering colorings $c$ which are constant on the levels of the tree $2^{{<}\omega}$ 
we can easily see that every Halpern--Läuchli family is a reaping family.
In fact, the following theorem states that this property characterizes reaping families which 
are indestructible by the Sacks forcing. 
The theorem is basically a compilation of results of 
Miller, Eisworth and Yiparaki.
We will provide the proof for the sake of completeness. 

\begin{theorem}[Eisworth, Miller~\cite{Miller(s)}, Yiparaki~\cite{Yiparaki}]\label{thm:HL-char}
	For a family $\mathcal R \subset {[\omega]}^\omega$ the following conditions are equivalent.
	\begin{enumerate}
		\item\label{thm:HL:cond-HL} $\mathcal R$ is a Halpern--Läuchli family.
		\item\label{thm:HL:cond-Sacksind} $\mathcal R$ is a reaping family in every generic extension via the Sacks forcing.
		\item\label{thm:HL:cond-existsext} There is an extension $W \supset V$ of the universe $V$ 
			such that $W \setminus V \cap \pw{\omega} \neq \emptyset$ and $\mathcal R$ is a reaping family in $W$.
		\item\label{thm:HL:cond-HLevrT} For every $p \in \mathbf S$ and $c \colon p \to 2$ 
			there is $q \in \mathbf S$, $q \leq p$ and $A \in \mathcal R$ 
			such that $q \restriction A$ is constant.
		\item\label{thm:HL:cond-intbr} For every $p \in \mathbf S$ 
			there is  $q \in \mathbf S$, $q \leq p$ and $A \in \mathcal R$ 
			such that either $A \subseteq x$ for each $x \in [q]$ 
			or $A \cap x = \emptyset$ for each $x \in [q]$.
	\end{enumerate}
\end{theorem}
\begin{proof}
	Implications (\ref{thm:HL:cond-Sacksind})$\Rightarrow$(\ref{thm:HL:cond-existsext}) 
	and (\ref{thm:HL:cond-HLevrT})$\Rightarrow$(\ref{thm:HL:cond-HL}) are trivial. 
	To prove (\ref{thm:HL:cond-Sacksind})$\Rightarrow$(\ref{thm:HL:cond-HLevrT}) 
	consider a Sacks indestructible reaping family $R$, $p \in \mathbf S$ and 
	$c \colon p \to 2$. Let ${r}_{\mathbf S} \in [p]$ be a Sacks generic real 
	and define $x = \os n \in \omega \mid c \left({r}_{\mathbf S}\restriction n\right) = 1\cs$. 
	Since $\mathcal R$ is a reaping family in $V[{r}_{\mathbf S}]$ 
	there is $A \in \mathcal R$ and $q \leq p$ such that either 
	$q \Vdash A \subseteq \dot{x}$ or $q \Vdash A \cap \dot{x} = \emptyset$. 
	In any case $c$ is constant on $q \restriction A$. 

	To see (\ref{thm:HL:cond-HLevrT})$\Rightarrow$(\ref{thm:HL:cond-intbr}) 
	just consider the function $c \colon p \to 2$ defined by 
	$c(t) = t(\card{t}-1)$ for $t \in p$.
	For (\ref{thm:HL:cond-intbr})$\Rightarrow$(\ref{thm:HL:cond-Sacksind})
	first notice that~(\ref{thm:HL:cond-intbr}) implies that $\mathcal R$ is a reaping family.
	Suppose $r \subset \omega$ is a real in a Sacks generic extension $W \supset V$, 
	we need to prove that $\mathcal R$ reaps $r$. If $r \in V$, we are done. 
	If $r \in W \setminus V$, then $r$ itself is a Sacks generic real over $V$.
	Now it is sufficient to notice that~(\ref{thm:HL:cond-intbr}) states that 
	$\mathbf S$ contains a dense set of conditions $q$ for which there is $A \in \mathcal R$ 
	such that $q$ forces that the generic real either contains $A$ or is disjoint from $A$.

	For (\ref{thm:HL:cond-existsext})$\Rightarrow$(\ref{thm:HL:cond-intbr}) let $p \in \mathbf S$ 
	and suppose $W \supset V$ is as in~(\ref{thm:HL:cond-existsext}). 
	As $W$ contains a new real, there exists $x \in ([p] \cap W) \setminus V$ and 
	$A \in \mathcal R$ such that $A \subseteq x$ or $A \cap x = \emptyset$. 
	If $A \subseteq x$ let $\bar{q} = \os y \in [p] \mid A \subset y \cs$, 
	if $A \cap x = \emptyset$ let $\bar{q} = \os y \in [p] \mid A \cap y = \emptyset \cs$. 
	In any case, $\bar{q} \in V$ is a closed subset of $[p]$ and $x \in \bar{q}$. 
	Thus $\bar{q}$ is uncountable and there exists a $q \in \mathbf{S}$, $q \leq p$ 
	such that $[q] \subseteq \bar{q}$. The condition $q$ is as required in~(\ref{thm:HL:cond-intbr}).

	It remains to show that (\ref{thm:HL:cond-HL})$\Rightarrow$(\ref{thm:HL:cond-Sacksind}).
	Let $\dot{x}$ be a Sacks name for a subset of $\omega$ and let $p \in \mathbf S$ 
	be a condition. 
	Using the usual fusion argument we can recursively 
	construct an infinite set $K \in {[\omega]}^\omega$ 
	and a condition $q \in \mathbf S$, $q \leq p$ such that there is a tree isomorphism 
	$\varphi \colon 2^{{<}\omega} \to q \restriction K$ and for each $t \in 2^{{<}\omega}$ 
	the condition $q_{\varphi(t)}$ forces either $\card{t} \in \dot{x}$ or $\card{t} \notin \dot{x}$. 
	Define $c \colon 2^{{<}\omega} \to 2$ by $c(t) = 1$ iff 
	$q_{\varphi(t)} \Vdash \card{t} \in \dot{x}$.
	Since $\mathcal R$ is a Halpern--Läuchli family there is $o \in \mathbf S$ and $A \in \mathcal R$ 
	such that $c$ has constant value $i$ on $o \restriction A$. 
	Let $q'$ be the downwards closure of $\varphi[o]$, i.e.\ $q' \in \mathbf S$ and $q' \leq q$. 
	Now if $i = 1$, then $q' \Vdash A \subset \dot{x}$ and if $i = 0$, 
	then $q' \Vdash A \cap \dot{x} - \emptyset$.
\end{proof}

\begin{corollary}\label{cor:small->ind}
	Every reaping family $\mathcal R$ of size smaller that $\mathfrak c$ 
	is a Sacks indestructible reaping family.
\end{corollary}
\begin{proof}
	We will verify that $\mathcal R$ satisfies condition~($\ref{thm:HL:cond-intbr}$) of Theorem~\ref{thm:HL-char}. 
	Suppose $p \in \mathbf S$, for every $A \in \mathcal R$ let 
	$\bar p_0(A) = \os x \in [p] \mid A \cap x = \emptyset \cs$ and 
	$\bar p_1(A) = \os x \in [p] \mid A \subseteq x \cs$. 
	Note that $[p] = \bigcup \os \bar p_i(A) \mid i \in 2, A \in \mathcal R \cs$ 
	since $\mathcal R$ is a reaping family. As $\card{\mathcal R} < \mathfrak c$ 
	there is $i \in 2$ and $A \in \mathcal R$ such that $\bar p_i(A)$ is an uncountable closed set. 
	Consequently this set contains a perfect set, i.e.\ there is $q \in \mathbf S$ such that 
	$[q] \subseteq \bar p_i(A)$; $q$ and $A$ are as required.
\end{proof}

Yiparaki was also looking into possible cardinalities of Halpern--Läuchli families.
She introduced the following variation of the Halpern--Läuchli property.

\begin{definition}[Yiparaki~\cite{Yiparaki}]
	A family $\mathcal R \subset {[\omega]}^\omega$ is called 
	a \emph{Halpern--Läuchli family by levels}
	if for every $c \colon 2^{{<}\omega} \to 2$ there are 
	$p \in \mathbf S$ and $A \in \mathcal R$ such that 
	$c$ is constant on $p \restriction \os n \cs$  for each $n \in A$.
\end{definition}

The following two definitions are also appeared in~\cite{Yiparaki}.
\[\mathfrak{hlt} = \min \os \card{\mathcal R} \mid \mathcal R \text { is a Halpern–Läuchli family}\cs\]
\[\mathfrak{hlt'} = \min \os \card{\mathcal R} \mid \mathcal R \text { is a Halpern–Läuchli family by levels}\cs\]

\begin{proposition}[Yiparaki~\cite{Yiparaki}]\label{prop:Yip}
The following inequalities hold.\footnote{
	The cardinal $\mathfrak{r}_\sigmaup$ is a relative of $\mathfrak r$. As we will not work with this cardinal, 
	we refer the interested reader e.g. to~\cite{Blass-cardinals} for the definition. Let us just mention that
	it is unknown whether $\mathfrak{r}_\sigmaup = \mathfrak r$ in~$\mathsf{ZFC}$.
}
	\begin{enumerate}
		\item $\mathfrak r \leq \mathfrak{hlt} \leq \max\os\mathfrak d , \mathfrak{r}_\sigmaup \cs$
		\item $\mathfrak{hlt} = \max\os\mathfrak r, \mathfrak{hlt'}\cs$
	\end{enumerate}
\end{proposition}

Yiparaki asked whether $\mathfrak{hlt} = \mathfrak{hlt'}$.
We will answer this question in positive. 

\begin{theorem}
	$\mathfrak{hlt} = \mathfrak{hlt'} = \mathfrak{r}$
\end{theorem}

Let us start with $\mathfrak{hlt}$.

\begin{proposition}\label{prop:hlt=r}
	$\mathfrak{hlt} = \mathfrak{r}$
\end{proposition}
\begin{proof}
	If $\mathfrak{r} < \mathfrak{c}$, then the reaping family $\mathcal R$ of size $\mathfrak r$ 
	is Halpern--Läuchli by Corollary~\ref{cor:small->ind} and Theorem~\ref{thm:HL-char}. 
\end{proof}

\begin{lemma}\label{lem:not-reap}
	Let $\mathcal R \subset {[\omega]}^\omega$ be family of size less than $\mathfrak r$.
	There exists a family $\mathcal S \subset {[\omega]}^\omega$ of pairwise disjoint sets, 
	$\card{\mathcal S} = \omega$ such that $A \cap S$ is infinite for every $A \in \mathcal R$ and $S \in \mathcal S$.
\end{lemma}
\begin{proof}
	We can assume that $\mathcal R$ is closed with respect to finite modifications.  
	Construct $\mathcal S$ by repeating the following procedure.
	Since $\mathcal R$ is not a reaping family, there exists $S \in {[\omega]}^\omega$
	such that both sets $A \cap S$ and $A \setminus S$ are infinite for every $A \in \mathcal R$.
	Add the set $S$ into the family $\mathcal S$ which is being constructed and repeat the procedure 
	for $\omega \setminus S$ in place of $\omega$ and $\os A \cap S \mid A \in \mathcal R\cs$ in place of $\mathcal R$. 
	After $\omega$ many steps we get the desired family $\mathcal S$.
\end{proof}

\begin{proposition}
	$\mathfrak{hlt'} = \mathfrak{r}$
\end{proposition}
\begin{proof}
	Due to Propositions~\ref{prop:Yip} and~\ref{prop:hlt=r} it suffices to prove that 
	$\mathfrak {r} \leq \mathfrak{hlt'}$. 
	Let $\mathcal R \subset {[\omega]}^\omega$ be a family of size smaller than $\mathfrak r$, we will show that 
	$\mathcal R$ is not Halpern--Läuchli by levels. 
	Let $\mathcal S = \os S(t) \mid t \in 2^{{<}\omega}\cs$ be as in Lemma~\ref{lem:not-reap}. 
	We now define $c \colon 2^{{<}\omega} \to 2$ in the following way. 
	Given $s \in 2^{{<}\omega}$, if there exist $t \in 2^{{<}\omega}$ and $i \in 2$ such that 
	$\card{s} \in S(t)$ and $t\conc{i} \subseteq s$, then let $c(s) = i$. 
	Otherwise define $c(s)$ arbitrarily. 

	Suppose $A \in \mathcal R$ and $p \in \mathbf S$ is a Sacks tree with stem $t$. 
	Choose $n \in A \cap S(t)$ such that $n > \card{t}$. 
	Now for both $i \in 2$ there exist $s_i \in p$, $t\conc i \subseteq s_i$, and $\card{s_i} = n$. 
	Thus $c(s_i) = i$ for $i \in 2$; $c$ is not constant on $p \restriction \os n \cs$ and $n \in A$.
\end{proof}

\section{Halpern--Läuchli ideals}

We are going to show that many of the standard definable ideals are HL 
with one notable exception -- the density zero ideal $\mathcal Z$. 
Let us start with a simple observation that the HL property of ideals 
is preserved downwards in the Katětov order.

\begin{lemma}\label{lem:katetov-HLideals}
	Let $\mathcal{I, J} \subseteq \pw{\omega}$ be ideals, $\mathcal I \leq_{\mathrm K} \mathcal J$. 
	If $\mathcal J$ is a HL ideal, then so is $\mathcal I$.
\end{lemma}
\begin{proof}
	Let $f \colon \omega \to \omega$ be a Katětov morphism witnessing $\mathcal I \leq_{\mathrm K} \mathcal J$.
	Let $W \supset V$ be an extension such that $W \models \mathcal J^+ \text{ is a reaping family}$, 
	we show that the same holds for $\mathcal I^+$. 
	Let $X \in \pw{\omega} \cap W$, there is $A \in \mathcal J^+$ which reaps $f^{-1}[X]$. 
	Then $f[A] \in \mathcal I^+$ and $f[A]$ reaps $X$.
\end{proof}

It is easy to define examples of ideals which are not HL\@. 
Let $c \colon 2^{{<}\omega} \to 2$ be a map, 
for $p \in \mathbf S$ define 
\[H_c(p) = \os n \in \omega \mid c \text{ is constant on } p \restriction \os n \cs \cs.\]
Let $\mathcal I_c$ be the (possibly improper) ideal generated by $\os H_c(p) \mid p \in \mathbf S \cs$.
In fact, it is easy to see that these ideals are critical for the HL property.

\begin{observation}\label{obs:Ic}
	An ideal $\mathcal J$ is not HL if and only if there exist $c \colon 2^{{<}\omega} \to 2$ 
	such that $\mathcal I_c \subseteq  \mathcal J$.
\end{observation}

Notice that condition~(\ref{thm:HL:cond-HLevrT}) of Theorem~\ref{thm:HL-char} implies 
that an equivalent condition for Observation~\ref{obs:Ic} is also the existence of $c \colon 2^{{<}\omega} \to 2$ and 
$q \in \mathbf S$ such that the ideal $\mathcal I_c(q)$ generated by $\os H_c(p) \mid p \leq q \cs$ 
is contained in $\mathcal J$.

It turns out that many examples of ideals are HL\@. 
We start by a generalization of the result of Miller on Sacks indestructibility of P-ultrafilters~\cite{Miller-forcing}. 

\begin{proposition}\label{prop:P+}
	If an ideal $\mathcal I$ is P$^+$, then $\mathcal I$ is a HL ideal.
\end{proposition}
\begin{proof}
	Suppose $\mathcal I$ is P$^+$ and pass to a generic extension via the poset 
	$\mathbf Q = (\mathcal I^+, \subset^*)$. As $\mathbf Q$ is $\sigmaup$-closed, 
	the generic extension does not contain any new reals and the HL property is absolute 
	between $V$ and the generic extension. 
	Let $\mathcal U$ be the generic filter on $\mathbf Q$. 
	It is easy to see that $\mathcal U$ is a P-ultrafilter and 
	$\mathcal I \subseteq \mathcal U^*$. 
	And since P-ultrafilters are Sacks indestructible, $\mathcal U^*$ is HL and so is $\mathcal I$.
\end{proof}

\begin{corollary}
	All F$_\sigmaup$ ideals are HL\@.
\end{corollary}

We can use a similar argument to reason that the ideal $\mathcal G_\mathrm c$ of graphs 
which do not contain an infinite complete subgraph is HL\@.

\begin{proposition}\label{prop:Gc}
	The ideal $\mathcal G_\mathrm c$ is a HL ideal.
\end{proposition}
\begin{proof}
	We can assume that there exists a Ramsey ultrafilter $\mathcal U$. 
	For if not, pass to a generic extension with no new reals and a Ramsey ultrafilter (by $\pw{\omega}/\fin$).
	The HL~property is absolute between the ground model and the extension. 
	We will show that $\mathcal G_\mathrm c^+$ remains a reaping family after adding a Sacks real. 
	Suppose $X \subset {[\omega]}^2$ is a set in a Sack extension. 
	Since $\mathcal U$ is Sacks indestructible (as first proved by Baumgartner and Laver~\cite{Baumgartner-Laver}) 
	and remains Ramsey in the extension, 
	there exists $A \in \mathcal U$ such that ${[A]}^2$ reaps $X$. 
	Notice that ${[A]}^2 \in \mathcal G_\mathrm c^+ \cap V$.
\end{proof}

Lemma~\ref{lem:katetov-HLideals} now gives us:

\begin{corollary}
	The ideals $\fin\times\fin$, $\mathcal{ED}$, $\conv$ are HL\@.
\end{corollary}

For the next result we will need a game introduced by Laflamme~\cite{Laflamme-games}. 
Suppose $\mathcal I$ is an ideal on $\omega$. The game $G(\mathcal I)$ associated 
to $\mathcal I$ takes $\omega$ many rounds and proceeds as follows: 
At round $n$ player~I chooses $I_n \in \mathcal I$ and player~II 
responds by choosing $k_n \in \omega \setminus I_n$. 
\medskip
\begin{center}
\begin{tabular}{|l || l | l | l | l| l ||l|l|} 
\hline
player I & 
${I_0}\in \mathcal I$ & 
${I_1}\in \mathcal I$ & 
$\ldots$ &
${I_n}\in \mathcal I$ &
$\ldots$ &
$\os {k_n \mid n \in \omega} \cs \in \mathcal I$
\\ 
\hline
player II & 
$k_0 \notin I_0$ & 
$k_1 \notin I_1$ & 
$\ldots$ &
$k_n \notin I_n$ &
$\ldots$ &
$\os {k_n \mid n \in \omega} \cs \notin \mathcal I$
\\ 
\hline
\end{tabular} 
\end{center}
\medskip
Player~I wins if $\os {k_n \mid n \in \omega} \cs \in \mathcal I$, otherwise player~II wins.
We will use a result of Hrušák which is contained in the proof 
of the category dichotomy theorem~\cite{hrusak-combinatorics,Hrusak-order_on_ideals}.

\begin{proposition}[Hrušák~\cite{Hrusak-order_on_ideals}]
	Let $\mathcal I$ be an ideal, if player~I has a winning strategy in the game $G(\mathcal I)$, 
	then there exists $X \in \mathcal I^+$ such that $\mathcal{ED} \leq_{\mathrm K} \mathcal I \restriction X$.
\end{proposition}

\begin{proposition}
	If an ideal $\mathcal I$ is not HL, then there exists 
	$X \in \mathcal I^+$ such that $\mathcal{ED} \leq_{\mathrm K} \mathcal I \restriction X$.
\end{proposition}

Note that the ideal $\mathcal{ED}$ is a HL ideal.

\begin{proof}
	Assume than an ideal $\mathcal I$ does not fulfill the conclusion of the proposition 
	and consequently player~I does not have a winning strategy in the game $G(\mathcal I)$. 
	We will show that $\mathcal I$ is a HL ideal.
	Let $c \colon 2^{{<}\omega} \to 2$ be a function, 
	for every $s \in 2^{{<}\omega}$ let $I(s)$ be the set of all $n \in \omega$, $n > \card{s} + 1 $ 
	such that there exists at most one $t \in 2^n$, $t \supset s$, and $c(t) = 0$. 

	We will consider two cases; first assume there is $s \in 2^{{<}\omega}$ such that $I(s) \in \mathcal I^+$. 
	Choose $A = \os a(n) \mid n \in \omega \cs \in \pw{I(s)} \cap \mathcal I^+$ 
	such that  $a(n) + 1 <a({n+1})$ for each $n \in \omega$.
	We will recursively construct a tree $S = \os s(t) \in 2^{{<}\omega} \mid t \in 2^{{<}\omega} \cs$.

	Put $s(\emptyset) = s$ and choose $s(\langle\, 0 \,\rangle)$, $s(\langle\, 1 \,\rangle) \in 2^{a(0)}$ 
	as two different extensions of $s$ such that $c(s(\langle\, 0 \,\rangle)) = c(s(\langle\, 1 \,\rangle)) = 1$. 
	This is possible since $s$ has at least $4$ extensions in $2^{a(0)}$ 
	and $c$ has value $0$ in at most one of these extensions. 
	Similarly, if $s(t) \in 2^{a(n)}$ is constructed for $t \in 2^{n+1}$, 
	choose $s(t\conc i) \in 2^{a(n+1)}$, $s(t\conc i) \supset s(t)$, $c(s(t\conc i)) = 1$ for $i \in 2$; 
	two different extensions of $s(t)$. 
	Finally let $p$ be the downwards closure of $S$. The construction implies that $p$ is a Sacks tree 
	and $c$ has constant value $1$ on $p \restriction A$.

	For the other case assume that $I(s) \in \mathcal I$ for each $s \in 2^{{<}\omega}$. 
	We will define a strategy for player~I in the game $G(\mathcal I)$. 
	While playing, player~I simultaneously constructs a tree 
	$S = \os s(t) \in 2^{{<}\omega} \mid t \in 2^{{<}\omega} \cs$. 
	She starts by declaring $s(\emptyset) = \emptyset$, and in general, 
	before the round $n$ of the game is played, player~I had already constructed $\os s(t) \mid t \in 2^n \cs$. 
	In round $n$ the move of player~I is the set $\bigcup \os I(s(t)) \mid t \in 2^n \cs \in \mathcal I$. 
	Suppose the response of player~II is $k_n$. 
	Since $k_n \notin I(s(t))$ for each $t \in 2^n$, 
	player~I can choose a set of pairwise different sequences 
	$\os s(t) \mid t \in 2^{n+1} \cs \subset 2^{k_n}$ such that 
	$s(t) \subset s(t\conc i)$ and $c(s(t\conc i)) = 0$ for each $t \in 2^n $ and $i\in 2$. 

	As the described strategy of player~I cannot be winning, 
	we can assume that a match of $G(\mathcal I)$ was played, 
	player~I followed the strategy, constructed tree $S$, and lost; 
	$K = \os k_n \mid n \in \omega\cs \in \mathcal I^+$. 
	Let $p$ be the downwards closure of $S$.  
	The construction again implies that $p$ is a Sacks tree 
	and $c$ has constant value $0$ on $p \restriction (K \setminus 1)$.
\end{proof}

\begin{corollary}
	The ideal $\nwd$ is HL\@.
\end{corollary}
\begin{proof}
	For every $X \in \nwd^+$ the ideal $\nwd \restriction X$ is isomorphic to $\nwd$ 
	and $\mathcal{ED} \not\leq_{\mathrm K} \nwd$, see~\cite{Hrusak-order_on_ideals}.
\end{proof}

We do have a similar result for the ideal $\conv$. 
Let us first recall a theorem proved by Meza Alcántara.

\begin{theorem}[Meza Alcántara.~\cite{Meza-thesis}, see also~\cite{Hrusak-etal_Ramsey}]\label{thm:conv-char}
	For an ideal $\mathcal I \subset \pw{\omega}$ the following are equivalent.
	\begin{enumerate}
		\item $\conv \leq_{\mathrm K} \mathcal I$
		\item There is a countable family $\mathcal X \subset {[\omega]}^\omega$ 
		such that for each $Y \in \mathcal I^+$ there is $X \in \mathcal X$ which splits $Y$.
	\end{enumerate}
\end{theorem}

\begin{proposition}
	If an ideal $\mathcal I$ is not HL, then $\conv \leq_{\mathrm K} \mathcal I$.
\end{proposition}

Let us again note that $\conv$ itself is a HL ideal.

\begin{proof}
	Suppose $\conv \nleq_{\mathrm K} \mathcal I$ and let $c \colon 2^{{<}\omega} \to 2$ be any map. 
	We will show that $\mathcal I_c \cap \mathcal I^+ \neq \emptyset$, i.e.\ $\mathcal I$ is HL\@.
	Let $\mathcal X \subset 2^\omega$ be a countable dense set. 
	For $x \in \mathcal X$ and $i \in 2$ let $K_i(x) = \os n \in \omega \mid c(x \restriction n) = i \cs$. 
	Theorem~\ref{thm:conv-char} implies that there is $Y \in \mathcal I^+$ 
	such that for every $x \in \mathcal X$ there is $i(x) \in 2$ such that $Y \subseteq^* K_{i(x)}(x)$.
	For $i \in 2$ let $\mathcal X_i  = \os x \in \mathcal X \mid x(i) = i\cs$, at least one of these sets 
	has to be somewhere dense, assume that $\mathcal X_0$ is dense above $s$. 

	We will recursively construct two Sacks trees $p^{\mathrm e}, p^{\mathrm o}$ with stems extending $s$ 
	and an increasing sequence $\os k(n) \in \omega \mid n \in \omega\cs$.
	Start by choosing $k(0)$ such that there exists $x \in \mathcal X_0$, $s \subset x$, 
	and $Y \subseteq K_{0}(x) \setminus k(0)$, 
	and declare $x \restriction k(0) \in p^{\mathrm e} \cap p^{\mathrm o}$. 

	If $k(n)$, $p^{\mathrm e} \cap 2^{{\leq}k(n)}$ and $p^{\mathrm o} \cap 2^{{\leq}k(n)}$ are already 
	constructed and $n \in \omega$ is even, choose 
	$k(n+1)$, $p^{\mathrm e} \cap 2^{{\leq}k(n+1)}$ and $p^{\mathrm o} \cap 2^{{\leq}k(n+1)}$ such that:
	\begin{enumerate}
		\item\label{cond:branch} Every $t \in p^{\mathrm e} \cap 2^{k(n)}$ has at least two extensions in 
		$p^{\mathrm e} \cap 2^{k(n+1)}$.
		\item\label{cond:prepare} For every $t \in p^{\mathrm e} \cap 2^{k(n+1)}$ there exists $x \in \mathcal X_0$, 
		$t \subset x$ such that $Y \subseteq K_{0}(x) \setminus k(n+1)$.
		\item\label{cond:monochrom} Every $t \in p^{\mathrm o} \cap 2^{k(n)}$ 
		has an extension in $p^{\mathrm o} \cap 2^{k(n+1)}$ and 
		for every $t \in p^{\mathrm o} \cap 2^{k(n+1)}$ and $m \in Y \cap [k(n), k(n+1))$
		is $c(t\restriction m) = 0$.
	\end{enumerate}
	In case $n$ is odd, choose 
	$k(n+1)$, $p^{\mathrm e} \cap 2^{{\leq}k(n+1)}$ and $p^{\mathrm o} \cap 2^{{\leq}k(n+1)}$ in a similar way, 
	just switch the requirements for $p^{\mathrm e}$ and $p^{\mathrm o}$. 

	For even $n$, to choose suitable elements of $p^{\mathrm e} \cap 2^{k(n+1)}$ 
	fulfilling~(\ref{cond:branch}) and~(\ref{cond:prepare})
	it is sufficient to use the density and the definition of $\mathcal X_0$.
	To choose suitable elements of $p^{\mathrm e} \cap 2^{k(n+1)}$ we can use the sets $x$ 
	required in~(\ref{cond:prepare}) of the previous step of the construction.
	If $n$ is odd, the situation is analogous.

	Finally, we have that $Y \subseteq H_c(p^{\mathrm e}) \cup H_c(p^{\mathrm o}) \in \mathcal I_c$ and we are done.
\end{proof}

For some ideals the verification of the HL~property gets more interesting. 

\begin{theorem}
	The ideal $\mathcal{SC}$ is a HL ideal.
\end{theorem}
\begin{proof}
	Let $c \colon 2^{{<}\omega} \to 2$ be any map, we will find a Sacks tree $q$ 
	such that $H_c(q) \in \mathcal{SC}^+$. 
	For $\ell \in \omega$ let $\mathrm I_\ell$ be the set of all intervals of length $\ell$, 
	$\mathrm I_\ell = \os [m, m+\ell) \mid m \in \omega \cs$.
	For $n, \ell \in \omega$, $n \leq \ell$ let 
	\[B^\ell_n = \os p \in \mathbf S \mid \exists^\infty b\in \mathrm I_\ell, n \leq \card{b \cap H_c(p)} \cs.\]
	Note that these sets are downwards closed in $\mathbf S$. 
	We want to construct $q \in \mathbf S$ such that for each $n \in \omega$ there is 
	$\ell \in \omega$ such that $q \in B^\ell_n$, 
	this will demonstrate that $H_c(q) \in \mathcal{SC}^+$. 

	We first prove that $B^\ell_\ell$ is dense in $\mathbf S$ for each $\ell \in \omega$.
	As in the proof of Proposition~\ref{prop:Gc} we can assume that there exists 
	a Ramsey ultrafilter~$\mathcal{U}$.
	Take any $p \in \mathbf S$ and let $g \in [p]$ be a Sacks generic real over $V$. 
	In $V[g]$ define a map $d\colon \omega \to 2^\ell$ by
	$d(m)(i) = c(g(m+i))$. 
	Since $\mathcal U$ generates an ultrafilter in $V[g]$, 
	there is $U \in \mathcal U$, $z \in 2^\ell$, and $r \in \mathbf S$, $r \leq p$ such that 
	$r \Vdash d(m) = z \text{ for each } m \in U$. 
	Thus necessarily $[m, m+\ell) \subset H_c(r)$ for each $m \in U$, 
	and consequently $r \in B^\ell_\ell$.

	Next we want to get set up for a fusion construction of the desired tree $q$. 
	We will say that a Sacks tree $p$ is \emph{complete} if for every $k, n \in \omega$ 
	the following condition holds:
	\begin{itemize}
		\item[$\star(p,k,n)$] For every set $\os p_j \leq p \mid j \in k\cs$ of conditions in $\mathbf S$ 
		with mutually incompatible stems there exist 
		$\ell \in \omega$ and $\os q_j \leq p_j \mid j \in k\cs$ such that 
		$\bigcup \os q_j \mid j \in k\cs \in B^\ell_n$.
	\end{itemize}
	We will show that there exists a complete Sacks tree.

	If $2^{{<}\omega}$ is not complete, let $k \in \omega$ be minimal for which 
	there is $n \in \omega$ such that $\star(2^{{<}\omega},k,n)$ fails. 
	Let $\os p_j \leq p \mid j \in k\cs$ be the set of conditions witnessing this failure. 
	Note that $k > 1$ since $B_n^n$ is dense. 
	We will prove that $p_0$ has to be complete. 
	Fix $k',n' \in \omega$ and $\os r_j \leq p_0 \mid j \in k'\cs$ 
	for verifying  $\star(p_0,k',n')$. 
	The minimality of $k$ implies $\star(2^{{<}\omega},k-1,n'+nk')$, 
	in particular there exist $\ell \in \omega$ and 
	$\os q_j \leq p_j \mid j \in k \setminus 1 \cs$ such that 
	$q' = \bigcup \os q_j \mid j \in k \setminus 1\cs \in B^\ell_{n'+nk'}$. 
	This means that there exist $b \in {[\ell]}^{n'+nk'}$, $e \colon b \to 2$ 
	and $A \in {[\omega]}^\omega$ such that 
	$c(t) = e(i)$ for every $t \in q' \restriction \os a + i\cs$ such that
	$a \in A$ and $i \in b$.

	Suppose that $g$ is a Sacks generic real. 
	In $V[g]$ define a map $d\colon \omega \to 2^b$ by
	$d(m)(i) = c(g(m+i))$. 
	Assume again that $\mathcal U$ is a Ramsey ultrafilter such that $A \in \mathcal U$. 
	Since $\mathcal U$ generates an ultrafilter in every Sacks extension, 
	we can find $U \in \mathcal U$, $U \subset A$, conditions $r'_j \leq r_j$, and functions
	$z_j \in 2^b$ such that for each $j \in k'$ we have 
	$r'_j \Vdash d(m) = z_j \text{ for each } m \in U$.
	For $j \in k'$ let $x_j = \os i \in b \mid z_j(i) = e(i) \cs$. 
	Since $\os p_j \mid j \in k\cs$ witnesses the failure of $\star(2^{{<}\omega},k,n)$, 
	it has to be the case that $\card{x_j} < n$ for each $j \in k'$, 
	otherwise $r'_j$ would work as $q_0$ together with $\os q_j \mid j \in k \setminus 1 \cs$ 
	for $\star(2^{{<}\omega},k,n)$. 
	Let $y = b \setminus \bigcup \os x_j \mid j \in k' \cs$. 
	We have $\card{y} \geq n'+nk' - nk' = n'$. 
	Note that if $i \in y$, and $j, j' \in k'$ then $z_j(i) \neq e(i) \neq z_{j'}(i)$, 
	i.e.\ $z_j(i) = z_{j'}(i)$. 
	Thus $\ell$ and $\os r'_j \mid j \in k'\cs$ are as required in $\star(p_0,k',n')$.

	We proved that there is a complete Sacks tree $p$. 
	Using a standard fusion construction we can now build $q$ as a perfect subtree of $p$ 
	by infinitely refining $p$ and making sure that for each $n \in \omega$ there 
	exists $\ell$ such that $q \in B^\ell_n$.
\end{proof}

We need a couple of auxiliary results before we can deal with the ideal $\tr({\mathscr N})$.
The following proposition is a characterization of the Sacks property, see~\cite{Miller-measure&category,Bartosz-Judah-reals}.

\begin{proposition}\label{prop:sacks-preserv-measure}
	A forcing $\mathbf P$ has the Sacks property if and only if it strongly preserves outer Lebesgue measure. 
	That is if $V[G]$ is a generic extension via $\mathbf P$ and $O \in V[G]$ is an open subset of $2^\omega$
	such that $\muup(O) < \varepsilon$, then there exists an open set 
	$U \in V$, $\muup(U) < \varepsilon$ such that $O \subseteq U$.
\end{proposition}

\begin{lemma}\label{lem:improve_tree}
	Let $B \subset 2^\omega$ be a Borel set of positive measure. 
	There exists a tree $r$ such that $[r] \subseteq B$ and 
	$\muup(r_s) > 0$ for each $s \in r$.
\end{lemma}
\begin{proof}
	Let $C \subseteq B$ be a closed such that $\muup(C) > 0$.
	Define \[X = \os s \in 2^{{<}\omega} \mid \muup([s] \cap C) = 0 \cs\]
	and let $D = C \setminus \bigcup \os [s] \mid s \in X \cs$. 
	$D$ is a closed set and for every $t \in 2^{{<}\omega}$ either 
	$D \cap [t] = \emptyset$ or $\muup(D \cap [t] )> 0$.
	The tree $r = \os x \restriction n \mid x \in D, n \in \omega \cs$ is as desired.
\end{proof}

\begin{theorem}
	The ideal $\tr({\mathscr N})$ is a HL ideal.
\end{theorem}
\begin{proof}
	Let $\dot a$ be an $\mathbf S$-name for a subset of $2^{{<}\omega}$. 
	We will prove that $\dot a$ is forced to be reaped by an element of ${\tr({\mathscr N})}^+ \cap V$. 
	Let $p \in \mathbf S$ be any condition. 
	
	\begin{case}
		There is $q \in \mathbf S$, $q \leq p$ and $b \subseteq 2^{{<}\omega}$
		such that $q \Vdash \muup(\piup(\dot a \cap b)) < \muup(\piup(b))$.
	\end{case} 
	The condition $q$ forces that there is an open set $O$ such that
	$\piup(\dot a \cap b) \subseteq \dot O$ and $\muup(\dot O) < \muup(\piup(b))$. 
	Proposition~\ref{prop:sacks-preserv-measure} gives us an open set $U \in V$, 
	$\muup(U) < \muup(\piup(b))$ such that 
	there is a condition $q_0 \leq q$ such that $q_0 \Vdash \dot O \subseteq U$. 
	Using Lemma~\ref{lem:improve_tree} we can find $r$ such that $[r] \subset \piup(b) \setminus U$ 
	and $\muup(r_s) > 0$ for each $s \in r$.
	Find $q_1 < q_0$ and $s \in r$ such that $q_1 \Vdash r_s \cap (\dot a \cap b) \text{ is finite}$. 
	We have that $r_s \cap b \in {\tr(\mathscr N)}^+$ and $q_1$ 
	forces this set to be almost disjoint with $\dot a$.
	
	\begin{case}
		$p \Vdash \muup(\piup(\dot a \cap b)) = \muup(\piup(b))$ for every $b \subseteq 2^{{<}\omega}$.
	\end{case}
	In particular, $p$ forces that $\piup(\dot a)$ has full measure.
	\begin{claim*}
		For every $m,n \in \omega$ and $\os p_i \leq p \mid i \in n \cs$ 
		there exist $\os q_i \leq p_i \mid i \in n \cs$ and a finite antichain 
		$b \in 2^{{<}\omega} \setminus 2^{{<}m}$ such that 
		$1 - \frac{1}{m} < \sum \os 2^{-\card{s}} \mid s \in b\cs$ and 
		$q_i \Vdash b \subseteq \dot a$ for every $i \in n$.
	\end{claim*}
	We prove the claim by induction on $n \in \omega$. 
	For $n = 1$ the claim holds since $\piup(\dot a)$ is forced to have full measure. 
	Suppose the claim holds for $n$, we will prove it for $n+1$. 
	Fix $m \in \omega$ and $\os p_i \leq p \mid i \in n+1 \cs$. 
	We can recursively find descending sequences $\os p_i^k \mid k \in \omega \cs$ 
	for $i \in n$ and finite antichains 
	$\os b_k \mid k \in \omega \cs$ 
	such that
	\begin{enumerate}
		\item $p_i^0 = p_i$ for $i \in n$,
		\item $b_k \subseteq 2^{{<}\omega} \setminus 2^{{<}k}$ for $k \in \omega$,
		\item $1 - \frac{1}{k} < \sum \os 2^{-\card{s}} \mid s \in b_k\cs$ for $k \in \omega$, and
		\item $p_i^k \Vdash b_k \subset \dot a$ for $i \in n$.
	\end{enumerate}
	Let $b = \bigcup \os b_k \mid k \in \omega\cs$ and note 
	that $\muup(\piup(b)) = 1$, thus
	$p_n \Vdash \muup(\piup(\dot a \cap b)) = 1$. 
	There is $q_n \leq p_n$ and a finite antichain $c \subset b \setminus 2^{{<}m}$ such that 
	$q_n \Vdash c \subset \dot a$ and $1 - \frac{1}{m} < \sum \os 2^{-\card{s}} \mid s \in c\cs$. 
	Since $c$ is finite, there is $k \in \omega$ such that $p^k_i \Vdash c \subseteq \dot a$ for each
	$i \in n$ and the claim is proved.
	\claimdone

	Using the claim we can run a standard fusion construction; 
	construct $q \leq p$, $q \in \mathbf S$ and finite antichains $\os b_k \mid k \in \omega \cs$
	such that
	\begin{enumerate}
		\item $b_k \subseteq 2^{{<}\omega} \setminus 2^{{<}k}$ for $k \in \omega$,
		\item $1 - \frac{1}{k} < \sum \os 2^{-\card{s}} \mid s \in b_k\cs$ for $k \in \omega$, and
		\item $q \Vdash b_k \subseteq \dot a$ for each $k \in \omega$.
	\end{enumerate}
	Finally $b = \bigcup \os b_k \mid k \in \omega\cs \in {\tr(\mathscr N)}^+$ and 
	$q \Vdash b \subseteq \dot a$.
\end{proof}

We know only one notable example of an ideal without the HL property.

\begin{theorem}\label{thm:Z-isHL}
	The ideal $\mathcal Z$ of sets of asymptotic density $0$ is not HL\@.
\end{theorem}
A related result was proved by Stepr\={a}ns~\cite{Steprans-density}.
\begin{proof}
	Choose a slowly branching tree $p \in \mathbf S$,
	in particular $\card{ p \cap 2^{n+1}} = n$
	 for $n \in \omega \setminus 1$.
	I.e.\ at levels in the interval $(2^n, 2^{n+1}]$ the tree $p$
	has exactly $n$ branches. 
	Given $n \in \omega \setminus 1$ enumerate $p \cap 2^{2^{n+1}}$ as $\os s_j^n \mid j \in n\cs$. 
	Choose a function $c \colon p \to 2$ such that for each $n \in \omega \setminus 1$
	the function $b^n \colon (2^n, 2^{n+1}] \to 2^n$ defined by  
	$b^n (i) = (j \mapsto c(s_j^n \restriction i))$ is a bijection.
	(When we track what the function $c$ does on levels in the interval $(2^n, 2^{n+1}]$, 
	we find each combination of assigning $0$ and $1$ to the $n$ branches on exactly one level.)

	Notice that for $q \leq p$, $q \in \mathbf S$ and $n \in \omega\setminus 1$, 
	if $\card{q \cap {2^{n+1}}} = k$, then 
	$\card{H_c(q) \cap (2^n, 2^{n+1}]} = 2^{n-k+1}$. 
	Since $q$ is perfect, the number of branches $\card{q \cap {2^{n+1}}}$ 
	diverges to infinity with increasing $n$, and consequently $H_c(q) \in \mathcal Z$. 
	We have $\mathcal I_c(p) \subseteq \mathcal Z$.
\end{proof}

\section{Halpern--Läuchli ultrafilters}

Our main interest is the existence of Sacks indestructible ultrafilters. 
The typical approach to constructing these ultrafilters is to attempt a recursive construction, 
starting with a filter (or a dual ideal) and enlarging it to construct an ultrafilter while 
trying to get the control over the~HL property. 
In fact, even our motivation for looking into HL ideals was this approach. 

The following is an immediate corollary of Theorem~\ref{thm:Z-isHL}.
\begin{corollary}
	Every Halpern--Läuchli ultrafilter is a $\mathcal Z$-ultrafilter.
\end{corollary}
\begin{proof}
	If an ultrafilter $\mathcal U$ is not a $\mathcal Z$-ultrafilter, then $\mathcal Z \leq_{\mathrm K} \mathcal U^*$. 
	As $\mathcal Z$ is not HL, neither is $\mathcal U^*$.
\end{proof}

The existence of $\mathcal Z$-ultrafilters is an open question. 
\begin{question}[Hrušák~\cite{hrusak-combinatorics}]
	Do $\mathcal Z$-ultrafilters exist in $\mathsf{ZFC}$?
\end{question}
A related result was proved by Gryzlov~\cite{Gryzlov} who showed that in $\mathsf{ZFC}$ 
there is an ultrafilter $\mathcal U$ such that for every injective function $f \colon \omega \to \omega$ 
there is $U \in \mathcal U$ such that $f[U] \in \mathcal Z$. 
This result was improved by Flašková~\cite{Flaskova-0-point} who proved that 
the same holds true for the summable ideal $\mathcal I_{\nicefrac{1}{n}}$ 
in place of $\mathcal Z$.

The situation is most favorable for constructing a given type of ultrafilter 
when any given filter generated by ${<} \mathfrak c$ many sets can be extended to an ultrafilter of this type. 
This phenomenon was studied by several authors, explicit treatment is in~\cite{Bendle-Flaskova}.

\begin{definition}[Brendle--Flašková~\cite{Bendle-Flaskova}]
	A class $\mathscr C$ of ultrafilters exists generically if every filter base of
    size ${<} \mathfrak c$ can be extended to an ultrafilter in~$\mathscr C$.
\end{definition}

Let us define a cardinal characterizing the generic existence of Halpern--Läuchli ultrafilters.

\begin{definition}
	\[\mathfrak{hl} = \min\os \cof(\mathcal I) \mid \mathcal I \text{ is a non-HL ideal}\cs\]
\end{definition}

Ketonen proved that P-ultrafilters exist if the dominating number $\mathfrak d = \mathfrak c$. 
In fact, since every ideal generated by~${<}\mathfrak d$ sets is a P$^+$-ideal,
Proposition~\ref{prop:P+} has the following corollary.

\begin{corollary}
	$\mathfrak d \leq \mathfrak{hl}$
\end{corollary}

\begin{question}
	Is $\mathfrak d = \mathfrak{hl}$ a theorem of $\mathsf{ZFC}$?
\end{question}

\begin{question}
	Is $\cov(\mathscr N) \leq \mathfrak{hl}$ a theorem of $\mathsf{ZFC}$?
\end{question}

\begin{proposition}
	Halpern--Läuchli ultrafilters exist generically if and only if $\mathfrak{hl} = \mathfrak{c}$.
\end{proposition}
\begin{proof}
	Suppose that $\mathfrak{hl} = \mathfrak{c}$ and $\mathcal I$ is an ideal generated by ${<} \mathfrak c$ 
	many sets. Then $\mathcal I$ is HL, and given any $c \colon 2^{{<}\omega} \to 2$ there is 
	$U \in \mathcal I^+$ and $p \in \mathbf S$ such that $p \restriction U$ is $c$-monochromatic. 
	Using this observation, we can extend any small filter in $\mathfrak c$ many steps into an 
	ultrafilter while making sure that the Halpern--Läuchli condition is fulfilled for every $c$.
	The other implication is immediate.
\end{proof}

\begin{corollary}~\label{cor:HLexist}
	If $\mathfrak u \leq \mathfrak{hl}$, then there is a Halpern--Läuchli ultrafilter.
\end{corollary}
\begin{proof}
	If $\mathfrak u < \mathfrak c$, then the ultrafilter witnessing 
	this is Halpern--Läuchli by Corollary~\ref{cor:small->ind}. 
	Otherwise $\mathfrak u = \mathfrak{hl} = \mathfrak c$.
\end{proof}

Our original hope was that it might be the case that $\mathfrak u \leq \mathfrak{hl}$ is just 
a theorem of $\mathsf{ZFC}$ and Corollary~\ref{cor:HLexist} could give an absolute 
result on the existence of Halpern--Läuchli ultrafilters.
However, this turns out not to be the case.

Brendle and Flašková~\cite{Bendle-Flaskova}, and independently Hong and Zhang~\cite{Hong-Zhang-invariants} 
introduced the cardinal characteristic $\mathfrak{ge}(\mathcal I)$ 
of an ideal $\mathcal I$ called the \emph{generic existence number}.
They observed that $\mathfrak{ge}(\mathcal I) = \mathfrak{c}$ is equivalent 
to the generic existence of $\mathcal I$-ultrafilters.
\begin{definition}
	Let $\mathcal I$ be an ideal on $\omega$. 
	\[\mathfrak{ge}(\mathcal I) = \min\os \cof(\mathcal J) \mid \mathcal I \subseteq \mathcal J, 
	\mathcal J \text{ is a proper ideal on } \omega  \cs \]
\end{definition}
The cardinal has been also studied in~\cite{Guzman-Hrusak-pospisil} where it was called 
the \emph{exterior cofinality}.

The definition gives us that $\non^*(\mathcal I) \leq \mathfrak{ge}(\mathcal I) \leq \cof(\mathcal I)$ 
for each ideal $\mathcal I$. 
Note also that if $\mathcal I$ is not a HL ideal, then $\mathfrak{hl} \leq \mathfrak{ge}(\mathcal I)$. 
We can combine Theorem~\ref{thm:Z-isHL} with a result of Fremlin~\cite{Fremlin} 
that $\cof(\mathcal Z) = \cof( \mathscr N)$ to get:

\begin{corollary}
	\[\mathfrak{hl} \leq \mathfrak{ge}(\mathcal Z) \leq \cof(\mathscr N) \]
\end{corollary}

It is also known that consistently $\mathfrak{ge}(\mathcal Z) < \cof(\mathscr N)$, 
see~\cite{Bendle-Flaskova}.
Notice that $\omega_1 = \mathfrak{hl} = \cof(\mathscr N) < \mathfrak u$ does hold in the Silver model
and there are no P-ultrafilters~\cite{silver-model}. 

\begin{question}
	Do Halpern--Läuchli ultrafilters exist in the Silver model?
\end{question}

The existence of P-points in the random is currently an open 
question~\cite{silver-model,Dow-random,FB-Hrusak-corrigendum}. 
What about Sacks indestructible ultrafilters?

\begin{question}
	Do Halpern--Läuchli ultrafilters always exist in the random model?
\end{question}

Let us overview the properties of Halpern--Läuchli ultrafilters and compare them with P-ultrafilters.

\begin{itemize}
	\item Halpern--Läuchli ultrafilters exist generically iff $\mathfrak{hl} = \mathfrak{c}$.
	\item Let $\mathcal U$ be an ultrafilter such that $\chi(\mathcal U) < \mathfrak c$. 
	Then $\mathcal U$ is Halpern--Läuchli.
	\item If $\mathfrak u \leq \mathfrak{hl}$, then Halpern--Läuchli ultrafilters exist.
\end{itemize}

And for P-ultrafilters we have:

\begin{itemize}
	\item P-ultrafilters exist generically iff $\mathfrak{d} = \mathfrak{c}$~\cite{Bendle-Flaskova}.
	\item Let $\mathcal U$ be an ultrafilter such that $\chi(\mathcal U) < \mathfrak d$. 
	Then $\mathcal U$ is a P-ultra\-filter~\cite{Ketonen}.
\end{itemize}

A set $\mathcal X \subset 2^\omega$ is said to have property (s) (Marczewski (Szpilrajn)) 
if for every perfect tree $p \in \mathbf S$ there exists $q \in \mathbf S$, $q < p$ 
such that either $[q] \subseteq \mathcal X$ or $[p] \cap \mathcal X = \emptyset$. 
Ultrafilters with property~(s) were studied by Miller~\cite{Miller(s)}. 
Condition~(\ref{thm:HL:cond-Sacksind}) of Theorem~\ref{thm:HL-char} directly 
implies that every Halpern--Läuchli ultrafilter does have property~(s). 
Miller showed~\cite{Miller(s)} that property~(s) is (consistently) really weaker than Sacks indestructibility 
and under $\cov(\mathscr M) = \mathfrak c$ constructed a non-Halpern--Läuchli ultrafilter 
with property~(s).

\begin{question}[Miller~\cite{Miller(s)}]
	Do ultrafilters with property~(s) exist in~$\mathsf{ZFC}$?
\end{question}

Miller also asked about ultrafilters indestructible with iterated Sacks $\mathbf S \iter \mathbf S$ forcing 
and countable product of Sacks forcing $\mathbf S^\omega$. 

\begin{question}[Miller~\cite{Miller(s)}]
Is there any difference between Halpern--Läuchli ultrafilters and ultrafilters indestructible 
with $\mathbf S \iter \mathbf S$? What about $\mathbf S^\omega$?
\end{question}

Finally there are examples of destructible ultrafilters which do have some nice properties. 
Various $\nwd$-ultrafilters which are not $\mathcal Z$-ultrafilters, and hence not Halpern--Läuchli were 
constructed assuming $\mathsf{CH}$ e.g.\ by Hong and Zhang~\cite{Hong-Zhang-Relations}. 
This demonstrates that unlike being a P-point, these properties do not imply indestructibility.

\begin{theorem}
	If $\mathfrak p = \mathfrak d$, then there exists a non-Halpern--Läuchli Q-ultrafilter.
\end{theorem}
\begin{proof}
	We will show that the construction of a non-Halpern--Läuchli ultrafilter due to 
	Yiparaki~\cite{Yiparaki} can be also used to produce a Q-ultrafilter. 
	For $n\in \omega$ let $X_n$ be the set of all partitions of $2^n$ into sets of size $2$ 
	and let $X = \bigcup \os X_n \mid n \in \omega\cs$. 
	Choose a system of pairwise disjoint infinite sets $\os A_x \mid x \in X \cs \subset {[\omega]}^\omega$, 
	such that $\min A_x \geq n$ for $x \in X_n$. 
	Define $c \colon 2^{{<}\omega} \to 2$ as follows. 
	If $t \in 2^k$, $k \in A_x$, $x \in X_n$, $t \restriction n = s_i \in  b = \os s_0, s_1 \cs \in x$, 
	and $s_0 <_{\mathrm{lex}} s_1$, then
	$c(t) = i$.
	Otherwise define $c(t)$ arbitrarily.

	Next we fix a system $\os D_\alpha \mid \alpha \in \mathfrak d \cs$ of interval partitions of $\omega$ 
	which is dominating, i.e.\ for every interval partition $E$ of $\omega$ there is $\alpha \in \mathfrak d$
	such that each element of $E$ intersects at most two intervals of $D_\alpha$, see e.g.~\cite{Blass-cardinals}. 
	Using the fact that $\mathfrak p = \mathfrak d$ we can recursively in $\mathfrak p$ many steps 
	construct a filter $\mathcal F$ which contains a selector for every $D_\alpha$, $\alpha \in \mathfrak d$ 
	and such that $\os A_x \mid x \in X \cs \subseteq \mathcal F^+$. 
	The filter $\mathcal F$ is a rare filter and every ultrafilter extending $\mathcal F$ 
	is a Q-ultrafilter. 

	We claim that $\mathcal F \cup \mathcal I_c^*$ generates proper filter, 
	i.e.\ for every for 
	$P \in {\mathbf S}^{{<}\omega}$
	is $\bigcup \os H_c(p) \mid p \in P \cs \notin \mathcal F$.
	Since $P$ is finite, there is $n \in \omega$ and $x \in X_n$ 
	such that for each $p \in P$ there is $b \in x$, $b \subset p$.
	Notice that the way $c$ was defined guarantees that 
	for every $k \in A_x$, if $b \in x \in X_n$, $b \subset p$, 
	then there are $t_0, t_1 \in p \cap 2^k$, $b = \os t_0\restriction n, t_1 \restriction n\cs$ 
	and consequently $c(t_0) \neq c(t_1)$, and in particular $k \notin H_c(p)$.
	We got $A_x \cap \bigcup \os H_c(p) \mid p \in P \cs = \emptyset$, 
	and $A_x \in \mathcal F^+$ gives us the desired conclusion.

	Every ultrafilter extending $\mathcal F \cup \mathcal I_c^*$ is a non-Halpern--Läuchli 
	Q-ultrafilter as it is disjoint with $\mathcal I_c$.
\end{proof}

\newpage
\bibliography{references}
\bibliographystyle{plain}

\end{document}